\documentclass[a4paper,11pt]{amsart}

\usepackage[english]{my-shortcuts}
\usepackage{srcltx,algorithm,algorithmic}
\usepackage{graphicx}   

\title[Coherence and Weak Null Space Property]{Small coherence implies 
	the weak Null Space Property}
\author{St\'ephane Chr\'etien} \thanks{National Physical Laboratory, 
	Hampton Road, TW11 0LW, UK. Email: stephane.chretien@npl.co.uk}
\author{Zhen Wai Olivier Ho} \thanks{Laboratoire de Math\'ematiques de Besan{\c c}on, Universit\'e de Franche Comt\'e, 25030 Besan{\c c}on,  France, 
zhen\_wai\_olivier.ho@univ-fcomte.fr}

\begin{document}
	\maketitle
	

	%
	%
	%

	\begin{abstract}
		In the Compressed Sensing community, it is well known that given a matrix $X \in \mathbb R^{n\times p}$ 
		with $\ell_2$ normalized columns, the Restricted Isometry Property (RIP) implies the Null Space Property (NSP). 
		It is also well known that a small Coherence $\mu$ implies a weak 
		RIP, i.e. the singular values of $X_T$ lie between $1-\delta$ and $1+\delta$ for "most" index subsets $T \subset \{1,\ldots,p\}$ with size governed by $\mu$ and $\delta$. In this short note, we show that a small Coherence implies a weak Null Space Property, i.e. 
		$\Vert h_T\Vert_2 \le C \ \Vert h_{T^c}\Vert_1/\sqrt{s}$ for most 
		$T \subset \{1,\ldots,p\}$ with cardinality $|T|\le s$. We moreover prove some singular value perturbation bounds that may also prove useful for other applications. 
	\end{abstract}
	
	
	\bigskip

	\section{Introduction}
	\subsection{Motivations}
	Compressed Sensing is a new paradigm for data acquisition which was 
	discovered in \cite{candes2006robust} and \cite{donoho2006compressed} and has had a paramount impact on 
	modern Signal Processing, Statistics, Applied Harmonic Analysis, 
	Machine Learning, to name just a few. The whole field started 
	after it was discovered that if $\beta$ is sufficiently sparse, 
	one could recover the support and sign pattern of a high dimensional vector 
	$\beta\in \mathbb R^p$ from just a few linear measurements
	\begin{align*}
	y & = X \beta +\epsilon,
	\end{align*}
	where $X \in \mathbb R^{n\times p}$, with $n \ll p$, by solving a simple convex programming problem of the form 
	\begin{align*}
	\min_{b\in \mathbb R^p} \ \frac12 \Vert y -Xb\Vert_2^2+ \lambda \Vert b\Vert_1.
	\end{align*}
	
	\vspace{.5cm} 
	
	\begin{center}
	\begin{minipage}{40em}
        \quad {\em In the remainder of this paper, we will assume that the columns of $X$ are 
	$\ell_2$ normalized.} 
    \end{minipage}
    \end{center}
    
    \vspace{.5cm}
	
	One condition implying that both support and sign pattern can be recovered is 
	called the Restricted Isometry Property (RIP) \cite{candes2008restricted}. More precisely, RIP is the property that for all index subset $T_0 \subset 
	\{1,\ldots,p\}$ with $|T_0|= s_0$, all the singular values of the submatrix $X_{T_0}$ whose columns are the columns of $X$ indexed by $T_0$, lie in the interval $(1-\delta,1+\delta)$.

	One key result relating RIP and recovery of the basic features of a sparse 
	vector is the fact that RIP implies the so-called Null Space Property, 
	which says that the kernel of $X$ does not contain any sparse vector. More precisely, the NSP is the property that for all $T_0\subset \{1,\ldots,p\}$ 
	with $|T_0|= s_0$, and for all $h \in {\rm Ker}(X)$, 
	\begin{align}
	\Vert h_{T_0}\Vert_2 & \le C \ \Vert h_{T_0^c}\Vert_1 /\sqrt{s_0}
	\label {nsp} 
	\end{align}  
	with $C \in (0,1)$. It is well known that the NSP is the key property behind sparse recovery using Basis Pursuit type of methods, whereas RIP is not. The main reason for introducing the RIP is that it provides a pedagogical step for proving the NSP in the case of random matrices. It was recently shown that the NSP can also be derived without the RIP for random design
	\cite{azais2014rice}. Thus, understanding more precisely what are the conditions on the design matrix for which we can obtain a kind of NSP is quite an important question in this field. 
	
	Some very interesting work has been published recently in order to test if the NSP or weaker version of this property hold for a given matrix using convex programming; see e.g. 
	\cite{d2011testing}. On the other hand, 
	one of the main drawbacks of the Restricted Isometry Property is that one cannot in general check if a given matrix $X$ satisfies it in polynomial 
	time. Therefore, RIP is usually not considered of practical interest. Another property often used in many sparse recovery problems is the property of small coherence.    
	
	The coherence of a matrix is an important quantity in the study of designs for sparse recovery is the coherence. It will be denoted by $\mu$, will be defined as
	\bea
	\mu & = & \max_{1\le k<l \le p} |\la X_{k},X_{l} \ra|.
	\eea
	If the columns are almost orthogonal, then, one usually expects that the performance of Basis Pursuit should be almost as good as in the orthogonal case. This have been rigorously  studied in e.g. \cite{candes2009near}. The main motivation for using the coherence is that it is conceptually intuitive and also very easy to compute. 
	
	On the other hand, it was also proved in \cite{tropp2008norms}, \cite[Theorem 3.2 and following comments]{candes2009near} that if a matrix $X$ has small coherence, then for 
	most index subsets $T_0$ with cardinal $|T_0|= s_0$, the singular values 
	of $X_{T_0}$ lie in the interval $(1-\delta,1+\delta)$ \footnote{the precise 
		result underpinning this statement 
		will be recalled in Section \ref{CohWeakRIP} below}. In other words, 
	small coherence implies a kind of weak RIP where the singular value 
	concentration property holds for most instead of all submatrices 
	with $s_0$ columns from $X$. However, such results, although conceptually very interesting do not address the main problem of proving NSP type properties. 
	
	\subsection{Goal of the paper}
	Our aim in the present paper is to understand better the role of the coherence for Compressed Sensing by understanding how a small coherence implies a weaker version of the Null Space Property. 
	The main result of the present work is the following. We prove that if 
	a matrix $X$ has small coherence, then, for most index subsets $T_0 
	\subset \{1,\ldots,p\}$ with cardinal $|T_0|= s_0$, and for all $h \in {\rm Ker}(X)$, \eqref{nsp} holds for some positive $C_{\mu}$. In other words, small coherence implies a kind of weak 
	Null Space Property which holds for most, instead of all, $T_0$ with $|T_0| = s_0$. 
	
	\subsection{Additional notation}
	For $T\subset \left\{1,\ldots,p\right\}$, we denote by $|T|$ the cardinal of $T$.
	Given a vector $x\in \R^{p}$, we set $x_{T}=(x_{j})_{j\in T}\in\R^{|T|}$.
	The canonical scalar product in $\R^p$ is denoted by $\la\cdot,\cdot\ra$. 
	
	For any matrix $A\in\mathbb R^{d_1\times d_2}$, we denote by $A^{t}$ its transpose. The set of symmetric real matrices is denoted by 
	$\mathbb S_n$. We denote by $\|A\|$ the operator norm of $A$. We use the Loewner ordering on symmetric real matrices: if $A\in \mathbb S_n$, 
	$0\preceq A$ denotes positive semi-definiteness of $A$, and $A\preceq B$ stands for $0\preceq B-A$. The singular values of $A$ will be denoted by $\sigma_{\max}(A)=$ $\sigma_{1}(A) \geq \cdots \geq \sigma_{\min \{d_1,d_2\}}$ $=\sigma_{\min}(A)$.

	\section{Background}
	In this section, we recall some well known previous results relating 
	coherence, singular value concentration, RIP and NSP. We begin 
	with some definitions. 
	\subsection{Weak NSP and weak RIP}
	\subsubsection{Weak Null Space Property}
	First, the weak-Null Space Property.
	\begin{defi}
		A matrix $X \in \mathbb R^{n\times p}$ satisfies the Weak Null Space Property weak-NSP($s$,$C$,$\pi$) if for at least a proportion 
		$\pi$ of all index subsets $T_0\subset \{1,\ldots,p\}$ 
		with $|T_0| = s_0$, and for all $h \in {\rm Ker}(X)$, 
		\begin{align}
		\Vert h_{T_0}\Vert_2 & \le C \ \Vert h_{T_0^c}\Vert_1/ \sqrt{s_0}.
		\label {nspdef} 
		\end{align}  
	\end{defi}

	Notice that when $\pi=1$, we recover the definition of the standard 
	Restricted Isometry Property.
	
	The main consequence of the weak Null Space Property is that exact recovery 
	holds for the basis pursuit problem. Since the work \cite{cohen2009compressed}, 
	this can be proved swiftly as follows. Let us first recall the framework: we  assume that $y=X\beta$, i.e. we are in the noise free setting and $\beta$ has support $T_0$ with $|T_0|\le s_0$. Then, we solve 
	\begin{align*}
	\min_{b\in \mathbb R^p} \quad \Vert b\Vert_1 \textrm{ s.t. } y=Xb.
	\end{align*} 
	Let $\hat{\beta}$ denote a minimizer. Then, we have 
	\begin{align*}
	\Vert \hat{\beta}\Vert_1 & \le \Vert \beta\Vert_1, 
	\end{align*}
	which gives 
	\begin{align}
	\Vert \hat{\beta}_{T_0^c}-\beta_{T_0^c}\Vert_1 & \le \Vert \hat{\beta}_{T_0}-\beta_{T_0}\Vert_1 + 2 \Vert \beta_{T_0^c} \Vert 
	\end{align}
	and thus, by the Cauchy-Schwartz inequality 
	\begin{align}
	\Vert \hat{\beta}_{T_0^c}-\beta_{T_0^c}\Vert_1 & \le \sqrt{s_0} \ \Vert \hat{\beta}_{T_0}-\beta_{T_0}\Vert_2 + 2 \Vert \beta_{T_0^c} \Vert
	\label{titi} 
	\end{align}
	Since $\beta$ has support $T_0$, we obtain that $\beta_{T_0^c}=0$. Using the fact that $\hat{\beta}-\beta$ lies in the kernel of $X$ and using  \eqref{nspdef}, 
	we obtain from \eqref{titi} that $\Vert\hat{\beta}_{T_0^c}-\beta_{T_0^c}\Vert_1=0$. 
	Using \eqref{nspdef} again, we conclude that $\Vert\hat{\beta}-\beta\Vert_1=0$, 
	i.e. exact recovery holds. More results of this type can be found in \cite{candes2008restricted} and \cite{eldar2012compressed}.

	\subsubsection{Weak Restricted Isometry Property}
	The weak-Restricted Isometry Property is the subject of the next definition. 
	\begin{defi} 
		A matrix $X \in \mathbb R^{n\times p}$ satisfies the Weak Restricted Isometry Property weak-RIP($s$,$\rho$,$\pi$) if for at least a proportion $\pi$ of all index subsets $T_0\subset \{1,\ldots,n\}$ with $|T_0| = s_0$,
		\begin{align}
		(1-\rho) \le \sigma_{\min}(X_{T_0}) & \le \cdots \le \sigma_{\max}(X_{T_0}) \le (1+\rho).
		\label{nspdef} 
		\end{align}	
	\end{defi} 
	
	Notice that when $\pi=1$, we recover the definition of the standard 
	Restricted Isometry Property. 
	
	\subsection{On the relationship between RIP and NSP} 
	One of the cornerstones of Compressed Sensing is the Null Space Property. It is well known that RIP implies NSP as stated in the next theorem. We will use the standard notations RIP($s_0$,$\rho$) for RIP($s_0$,$\rho$,1) and 
	NSP($s_0$,$C$) for NSP($_0s$,$C$,1).
	\begin{theo}{\bf \cite{candes2008restricted}}
		Any matrix $X \in \mathbb R^{n\times p}$ satisfying RIP($2s_0$,$\delta$) 
		satisfies NSP($s_0$,$C$) with $C\le \sqrt{2} (1+\delta)/(1-\delta)$. 
	\end{theo}
	
	\subsection{On the relationship between the Coherence and weak-RIP} 
	\label{CohWeakRIP}
	The first result relating small coherence with weak-RIP was established 
	by \cite{candes2009near} based on a result about column selection due to 
	Tropp \cite{tropp2008norms}. A refinement of this result is recalled in 
	the next theorem.
	\begin{theo}{\bf Chr\'etien and Darses \cite{chretien2012invertibility}} \label{main1}
		Let $r\in(0,1)$, $\alpha \geq 1$. Let us be given a full rank matrix $X\in \mathbb R^{n\times p}$ and a
		positive integer $s_0$, such that
		\bea
		\mu & \le &  \frac{r}{(1+\alpha)\log p}\\
		s_0 & \le & \frac{r^2}{(1+\alpha) e^2}\ \frac{p}{ \|X\|^{2} \log p }.
		\eea
		Let $T_0\subset \left\{1,\ldots,p\right\}$ 
		be a random support with uniform distribution on index sets satisfying $|T_0| = s_0$.
		Then the following bound holds:
		\bea \label{sing}
		\bP \left(\|X_{T_0}^tX_{T_0}-I \|\geq r \right) & \le & \frac{1944}{p^{\alpha}}.
		\eea
		\label{CDspl}
	\end{theo}
	
	This theorem was used in, e.g. \cite{chretien2014sparse} for a study 
	of the LASSO when the variance is unknown. It has been also used in remote sensing \cite{hugel2014remote}, in the study of Gaussian erasure channels \cite{ozccelikkale2014unitary},
	Kaczmarcz type methods for least squares \cite{needell2015randomized}, extentions of RIP 
	\cite{barg2015restricted}; see also \cite{foucart2013recovery}.
	
	\subsection{The Gershgorin bound}
	The Gershgorin theorem gives a bound on the operator norm as a function of 
	the coherence. More precisely, as discussed e.g. in \cite{bandeira2013road},  
	for each index subset $T \subset \{1,\ldots,p\}$ with cardinal $|T_0|=s_0$, 
	\begin{align}
	\Vert X_{T_0}^t X_{T_0} -I \Vert & \le \mu (s_0-1). 
	\label{gersh}
	\end{align} 
	Clearly, this result starts being useful when $\mu$ is much smaller than $s_0$. In 
	the application for the LASSO, it is often assumed that this indeed the case as in e.g.  \cite{candes2009near}.	
	
	\section{Main results: small coherence implies weak-NSP}
	
	In this section, we state and prove the main result of this paper, namely that small coherence implies weak-NSP. Our main theorem is the following. 
	\begin{theo}
			Let $X \in \mathbb R^{n\times p}$, $s_0\le n$ and $\alpha>0$. Assume that 
			\bea
		        s_0 & \le & \frac{1}{16 (1+\alpha) e^2}\ \frac{p}{ \|X\|^{2} \log p }.
		    \eea
			Let 
			$\mu$ denote the coherence of $X$. Let  
			\begin{align*}
			\varepsilon_{min}=\frac{\frac14 \ s^3_0 \ \mu^2 +s_0^{3/2} \ \mu }{ \ \left(3-4 s_0 \mu^2\right)}
			\end{align*}
			\begin{align*}
			\varepsilon_{max}=144 s_0^3\ \mu^2+72 s_0^{3/2}\mu. 
			\end{align*}
			Assume that 
		\begin{align*}
		\mu \le \min \left\{\frac{1}{\sqrt{288 s_0^{5/2}\left(2 s_0^{3/2}+1\right)}},\frac{1}{\sqrt{\frac32 s_0^4 +6 s_0^{5/2}+2 s_0}},
		\frac{1}{4(1+\alpha)\log p}\right\}.
		\end{align*}
		Then, the matrix $X$ verifies the weak-NSP($s_0$,$C$,$\pi$) with $\pi=1-1944/p^\alpha$ and 
		\begin{align*}
		    C & = \frac{1  + 3\ s_0 \ \left(\varepsilon_{max} +\varepsilon_{min}\right)}{ \lambda_1 - 3\ s_0 \ \varepsilon_{min} } .
		\end{align*}
		In particular, if 
		\begin{align}
		    \mu \le \min \left\{\frac{c_0}{s_0^{5/2}},
		\frac{1}{4(1+\alpha)\log p}\right\}
		\end{align}
		for some positive constant $c_0$,  
		then the matrix $X$ verifies the weak-NSP($s_0$,$C$,$\pi$) with $\pi=1-1944/p^\alpha$ and \begin{align*}
			\varepsilon_{min}= \frac{1}{4} \frac{c_0^2 s_0^{-2}/4 +c_0 s_0^{-1}}{1/2-c_0^2 s_0^{-4}}
			\end{align*}
			\begin{align*}
			\varepsilon_{max}= \frac{1}{4} \frac{144 s_0^{-1}c_0^2 +72 c_0 s_0^{-2}}{\lambda_1 -1}
			\end{align*}
		\label{main} 
        Then, the matrix $X$ verifies the weak-NSP($s_0$,$C$,$\pi$) with $\pi=1-1944/p^\alpha$ and 
		\begin{align}
		    C & = \frac{1  + \frac{3}{4} \left(\frac{c_0^2 s_0^{-1}/4 +c_0}{1/2 -c_0^2s_0^{-4} } + \frac{144 c_0^2 + 72c_0 s_0^{-1}}{\lambda_1-1}\right)}{1 - \frac{3}{4} \frac{c_0^2 s_0^{-1}/4 +c_0}{1/2 -c_0^2s_0^{-4} } }.
		\label{C} 
		\end{align}
	\end{theo}
	\begin{proof} 
	    Using Theorem \ref{CDspl}, for 
	    \bea 
	    \mu & \le & \frac{1}{4(1+\alpha)\log p}
	    \eea
	    with probability larger that $\pi$, an index subset $T_0$ with cardinality $s_0$ 
	    \bea
		s_0 & \le & \frac{1}{16 (1+\alpha) e^2}\ \frac{p}{ \|X\|^{2} \log p }.
		\eea
	    satisfies 
	    \begin{align}
	        \frac54 & \ge \lambda_1 
	        \ge \lambda_{s_0} \ge \frac34.
	    \label{un} 
	    \end{align}
	    where 
	    \begin{align}
	        \lambda_1 & := \lambda_1(X_{T_0}X_{T_0}^t)
	    \label{deux} 
	    \end{align}
	    and 
	    \begin{align}
	        \lambda_{s_0} & := \lambda_{s_0}(X_{T_0}X_{T_0}^t).
	    \label{trois} 
	    \end{align}
		Let $h \in {\rm Ker}(X)$ and let $T_0$ be a subset of $\{1,\ldots,p\}$ with cardinality $|T_0|=s_0$ verifying \eqref{un}, \eqref{deux} and 
		\eqref{trois}. 
		Define 
		\begin{enumerate}
			\item[(i)] $T_1$ as the index set of the 
			$s_0$ largest entries of $h_{T_0^c}$ in absolute value,  
			\item[(ii)] $T_2$ as the index set of the 
			$s_0$ largest entries of $h_{(T_0\cup T_1)^c}$ in absolute value,
			\item[(iii)] etc \ldots
		\end{enumerate}
		Let $J$ denote the number of subsets obtained in this process \footnote{The last set contains the remaining smallest terms in absolute value and may not contain $s$ terms}. Let $T=T_0\cup T_1$. 
		By \eqref{borninf} in Corollary \ref{cor}, we have that 
		\begin{align}
		\left( \lambda_{s_0} - 3\ s_0 \ \varepsilon_{min} \right)\ \Vert h_{T}\Vert_2^2 & \le \Vert X_{T} h_{T} \Vert_2^2. 
		\label{toto}
		\end{align}	
		Moreover, since $h$ belongs to the kernel of $X$, 
		\begin{align*} 
		\Vert X_T h_T\Vert_2^2 & = \left\vert \langle X_T h_T, X h \rangle
		-\langle X_T h_T, X_{T^c} h_{T^c} \rangle\right\vert , \\
		& = \left\vert\sum_{j=2,\ldots,J} \langle X_T h_T, X_{T_j} h_{T_j} \rangle
		\right\vert.
		\end{align*} 
		On the other hand, by Lemma \ref{scal}, we have for $j=2,\ldots,J$,
		\begin{align*}
		\langle X_{T} h_{T}, X_{T_j} h_{T_j} \rangle & \le \left( \lambda_1 + 3\ s_0 \ \varepsilon_{max} \right) \ \Vert h_T\Vert_2 \ \Vert h_{T_j} \Vert_2. 
		\end{align*}
		Therefore, 
		\begin{align*} 
		\Vert X_T h_T\Vert_2^2 & = \left\vert \sum_{j=2,\ldots,J} \langle X_T h_T, X_{T_j} h_{T_j} \rangle \right\vert \\
		& \le \sum_{j=2,\ldots,J} \left\vert  \langle X_T h_T, X_{T_j} h_{T_j} \rangle \right\vert \\
		& \le \left( 		\lambda_1-\lambda_{s_0}  + 3\ s_0 \ \left(\varepsilon_{max} +\varepsilon_{min}\right) \right) \ \Vert h_T\Vert_2 \ \sum_{j=2,\ldots,J}  \  \ \Vert h_{T_j} \Vert_2. \\
		\end{align*} 
		By Lemma \cite[Lemma A.4]{eldar2012compressed}, we get
		\begin{align}
		\sum_{j=2,\ldots,J} \Vert h_{T_j} \Vert_2 & \le \frac{\Vert h_{T_0^c} \Vert_1}{\sqrt{s_0}} 
		\end{align}
		and we can deduce that 
		\begin{align}
		\label{titi} 
		\Vert X_T h_T\Vert_2^2 
		& \le \left( 		\lambda_1-\lambda_{s_0}  + 3\ s_0 \ \left(\varepsilon_{max} +\varepsilon_{min}\right) \right) \ \Vert h_T\Vert_2 \ \frac{\Vert h_{T_0^c} \Vert_1}{\sqrt{s_0}}.
		\end{align} 
		Combined \eqref{titi} with \eqref{toto} gives 
		\begin{align*} 
		\Vert h_T \Vert_2 & \le 
		\frac{ 		\lambda_1-\lambda_{s_0}  + 3\ s_0 \ \left(\varepsilon_{max} +\varepsilon_{min}\right)}{ \lambda_{s_0} - 3\ s_0 \ \varepsilon_{min} } \ \frac{\Vert h_{T_0^c} \Vert_1}{\sqrt{s_0}}.
		\end{align*}

	\end{proof}
	
	The bound \eqref{C} on $C$ in this Theorem can be made arbitrarily small by taking $c_0$ accordingly sufficiently small.  
	
	\section{Conclusion}
	In this paper, we established a relationship between the coherence and a weak version of the Null Space Property for design matrices in Compressed Sensing. Our approach is based on perturbation theory and no randomness assumption on the design matrix is used to establish this property. We expect that this result will be helpful to study a larger class of designs than usually done in the literature. In a future paper, we will show that such bounds can be fruitfully applied to simplify the analysis of Robust PCA. 
	
	\appendix 
	
	\section{Technical lemm\ae} 
	\subsection{Some perturbation results}
	Perturbation after appending a column to a given matrix is a special 
	type of perturbation. A survey on this topic is \cite{chretien2014perturbation}. 
	\subsubsection{Background}
	Recall that for a matrix 
	$A$ in $\mathbb R^{n\times n}$, $p_A$ denotes 
	the characteristic polynomial of $A$. 
	\begin{lemm} {\bf Cauchy's Interlacing theorem}.
		If $A \in \mathbb R^{n\times n}$ is a symmetric matrix with 
		eigenvalues $\lambda_1 \geq \cdots \geq \lambda_n$ and associated eigenvectors $v_1$,\ldots,$v_n$, and $v \in \mathbb R^n$, then 	
		\begin{align}
		p_{A+vv^t}(x) & = p_A(x) \left(1-\sum_{i=1}^n \frac{\langle v,u_i \rangle^2}{x - \lambda_i} \right). 
		\end{align}
	\end{lemm} 
	The previous lemma states in particular that the eigenvalues of $A$ interlace those of $A+vv^t$. 
	\subsubsection{Appending one vector: perturbation of the smallest non zero eigenvalue}
	If we consider a subset $T_0$ of $\{1,\ldots,p\}$ and a submatrix $X_{T_0}$ of $X$, the problem of studying the eigenvalue perturbations resulting from appending a column $X_j$ to $X_{T_0}$, with $j\not\in T_0$ 
	can be studied using Cauchy's Interlacing Lemma as in the following result. 
	\begin{lemm} \label{lemme:2}
		Let $T_0\subset \{1,\ldots,p\}$ with $\vert T_0 \vert=s_0$ and $X_{T_0}$ a submatrix of $X$. Let  $\lambda_1\geq ... \geq \lambda_{s_0}$ be the eigenvalues of $X_{T_0}X^t_{T_0}$. 
		Let $\tilde \lambda_{s_0}\le \lambda_{s_0}$. 
		Assume that $\tilde\lambda_{s_0}<1-s_0\mu^2$, we have
		\begin{align*} 
		\lambda_{s_0+1} \left(X_{T_0}X_{T_0}^t+X_jX_j^t\right) 
		& \geq \tilde\lambda_{s_0}-\epsilon_{s_0,min}
		\end{align*}	
		with 
		\begin{align*} 
		\epsilon_{s_0,min} & = \frac12 \Bigg(	\frac{s^3_0 \ \mu^2 \ \Vert X_{T_0} \Vert^2+4s_0^{\frac32} \ \mu \ \Vert X_{T_0} \Vert \tilde \lambda_{s_0}}{2 \ \left(1-s_0 \mu^2- \tilde \lambda_{s_0}\right)}\Bigg).
		\end{align*}
	\end{lemm} 	
	\begin{proof}
		Setting $v=X_j$
		\begin{align*}
		A & = X_{T_0} X_{T_0}^t 
		\end{align*}
		we obtain that the smallest nonzero eigenvalue of $X_{T_0}X_{T_0}^t+X_jX_j^t$ is the smallest root $\rho_{\min}$ of 
		\begin{align*}
		f(x) & = 1-\sum_{i=1}^n \frac{\langle v,u_i \rangle^2}{x - \lambda_i}.
		\end{align*}
		Therefore, $\rho_{\min}$ is larger than the smallest positive root of    
		\begin{align*}
		\tilde f(x) & = 1-\frac{s_0 \ \gamma}{x - \tilde \lambda_{s_0}}
		-\frac{1-s_0 \mu^2}{x}
		\end{align*}
		for any upper bound $\gamma$ to $\langle v,u_i \rangle^2$ for $i=1,\ldots,s_0$. Thus,
		we find that 
		\begin{align}\label{eq:lemme2}
		\rho_{\min} & \geq \frac12 \left( s_0 (\gamma-\mu^2) + \tilde\lambda_{s_0} + 1- 
		\sqrt{s^2_0\gamma^2+2s_0\gamma\left(\tilde\lambda_{s_0} +1-s_0 \mu^2
			\right)+ \left(1-s_0 \mu^2-\tilde\lambda_{s_0} \right)^2}\right). 
		\end{align}
		As long as $1-s_0 \mu^2>\lambda_{s_0} $, we have 
		\begin{align*}
		\rho_{\min} & \geq \frac12 \left( s_0 (\gamma-\mu^2) + \tilde\lambda_{s_0} + 1- \left(1-s_0 \mu^2-\tilde \lambda_{s_0}\right)
		\sqrt{1+ \frac{s^2_0\gamma^2+2s_0\gamma (\tilde \lambda_{s_0} +1-s_0 \mu^2)}{\left(1-s_0 \mu^2-\tilde \lambda_{s_0} \right)^2}}\right). 
		\end{align*}
		Moreover, since $\sqrt{1+a}\le 1+\frac12 a$, we get 
		\begin{align*}
		\rho_{\min} & \geq \frac12 \Bigg( s_0 (\gamma-\mu^2) +\tilde\lambda_{s_0}
		+ 1- \left(1-s_0 \mu^2-\tilde \lambda_{s_0}\right)
		\Big(1+ \frac{s^2_0\gamma^2+2s_0\gamma (\tilde\lambda_{s_0}+1-s_0 \mu^2)}{2 \ \left(1-s_0 \mu^2-\tilde \lambda_{s_0}\right)^2}\Big)\Bigg)
		\end{align*}
		which gives 
		\begin{align}
		\rho_{\min} & \geq \tilde \lambda_{s_0} - \epsilon_{s_0,min}
		\label{lowbnd}
		\end{align} 
		with 
		\begin{align*} 
		\epsilon_{s_0,min} & = \frac12 \Bigg(	\frac{s^2_0\gamma^2+4s_0\gamma \tilde \lambda_{s_0}}{2 \ \left(1-s_0 \mu^2-\tilde \lambda_{s_0}\right)}\Bigg).
		\end{align*}
		Let us now find out a reasonable value of $\gamma$. Let $X_{T_0}=U_0\Sigma_0 V_0^t$ denote the singular value decomposition of $X_{T_0}$. We have 
		\begin{align*}
		\vert \langle X_j,u_{j_0} \rangle \vert & = \vert \langle X_j,X_{T_0}V_0\Sigma_0 e_{j_0} \rangle \vert \\
		& = \Vert X_{T_0}^t X_j \Vert_2 \Vert V_0\Sigma_0 e_{j_0} \Vert_2 \\
		& \le \sqrt{s_0} \ \mu \ \Vert X_{T_0} \Vert.  
		\end{align*}
		Therefore we can take 
		\begin{align*}
		\gamma & = \sqrt{s_0} \ \mu \ \Vert X_{T_0} \Vert. 
		\end{align*}
		Combining this result with \eqref{lowbnd}, we get the desired result.
	\end{proof}  
	
	\subsubsection{Appending one vector: perturbation of the largest eigenvalue}\label{lemme:3}
	For the largest eigenvalue, we obtain 
	\begin{lemm}
		Let $T_0\subset \{1,\ldots,p\}$ with $\vert T_0 \vert=s_0$ and $X_{T_0}$ a submatrix of $X$. Let  $\lambda_1\geq ... \geq \lambda_{s_0}$ be the eigenvalues of $X_{T_0}X^t_{T_0}$. Let  $\tilde \lambda_1\ge \lambda_1$, with $\tilde \lambda_1>1$. Then, we have 
		\begin{align*} 
		\lambda_{1} \left(X_{T_0}X_{T_0}^t+X_jX_j^t\right) 
		& \le \tilde \lambda_1+\epsilon_{s_0,\max}. 
		\end{align*}	
		with 
		\begin{align*} 
		\epsilon_{s_0,\max} & = \frac12 \Bigg(\frac{s^3_0\mu^2\|X_{T_0}\|^2+4s_0^{3/2}\mu \ \|X_{T_0}\|\tilde\lambda_1}{2(\lambda_1 -1)}\Bigg).
		\end{align*}
	\end{lemm}
	\begin{proof}
		Setting $v=X_j$
		\begin{align*}
		A & = X_{T_0} X_{T_0}^t 
		\end{align*}
		we obtain that the largest nonzero eigenvalue of $X_{T_0}X_{T_0}^t+X_jX_j^t$ is the largest root $\rho_{\max}$ of 
		\begin{align*}
		f(x) & = 1-\sum_{i=1}^n \frac{\langle v,u_i \rangle^2}{x - \lambda_i}.
		\end{align*}
		Therefore, $\rho_{\max}$ is smaller than the largest positive root of    
		\begin{align*}
		\tilde f(x) & = 1-\frac{s_0 \ \gamma}{x - \tilde\lambda_1}
		-\frac{1}{x}
		\end{align*}
		for any upper bound $\gamma$ to $\langle v,u_i \rangle^2$ for $i=1,\ldots,s_0$. Hence, we find that 
		\begin{align}\label{eq:lemme3}
		\rho_{\max} & \le \frac12 \left( s_0 \gamma + \tilde\lambda_1 + 1 + 
		\sqrt{s^2_0\gamma^2+2s_0\gamma\left(\tilde\lambda_1+1
			\right)+ \left(1 -\tilde\lambda_1\right)^2}\right). 
		\end{align}
		Since the columns of $X$ have unit $\ell_2$-norm, we have $1<\lambda_1$, and thus one obtains from \eqref{eq:lemme3} that
		\begin{align*}
		\rho_{\max} & \le \frac12 \left( s_0 \gamma + \tilde\lambda_1 + 1+ \left(\tilde\lambda_1-1\right)
		\sqrt{1+ \frac{s^2_0\gamma^2+2s_0\gamma (\tilde\lambda_1+1 )}{\left(\tilde\lambda_1-1\right)^2}}\right)
		\end{align*}
		which gives
		\begin{align*}
		\rho_{\max} \le \tilde\lambda_1 +\epsilon_{s_0,\max} 
		\end{align*}
		with 
		\begin{align*} 
		\epsilon_{s_0,\max} & = \frac12 \Bigg(\frac{s^2_0\gamma^2+4s_0\gamma \tilde \lambda_1}{2(\tilde\lambda_1-1)}\Bigg).
		\end{align*}
		We finally plug in the value of $\gamma$ found earlier in the proof of Lemma \ref{lemme:2} to get the desired result.
	\end{proof}
	
	\subsubsection{Successive perturbations}
	
	If we append $s_0$ columns successively, we obtain the following result.
	\begin{lemm}
		Let $T_0\subset \{1,\ldots,p\}$ with $\vert T_0 \vert=s_0$ and $X_{T_0}$ a submatrix of $X$. Let  $\lambda_1\geq ... \geq \lambda_{s_0}$ be the eigenvalues of $X_{T_0}X^t_{T_0}$. Let 
		$\tilde \lambda_1 \ge \lambda_1$ and $\tilde \lambda_{s_0} \le \lambda_{s_0}$. 
		Let $T_1\subset \{1,\ldots,p\}$ with $\vert T_1 \vert=s_1$ and $T_0 \cap T_1=\emptyset$. Assume 
		\begin{enumerate}
			\item $1-(s_0+s_1)\mu>\tilde \lambda_{s_0}> \eta$;
			\item $1<\tilde \lambda_1<2-\eta$;
			\item $s_1< \min\left(\frac{\tilde \lambda_{s_0}-\eta}{\varepsilon_{min}},\frac{2-\eta-\tilde \lambda_1}{\varepsilon_{max}}\right)$; \label{lemme4:h3}
		\end{enumerate}
		with
		\begin{align*}
		\varepsilon_{min}=\frac14 \Bigg(	\frac{s^3_0 \ \mu^2 \ \eta^2+4s_0^{3/2} \ \mu \ \eta^2}{ \ \left(1-s_0 \mu^2- \eta\right)}\Bigg)
		\end{align*}
		\begin{align*}
		\varepsilon_{max}=\frac14 \Bigg(\frac{(s_0+s_1)^3\mu^2(2-\eta)^2+4(s_0+s_1)^{3/2}\mu (2-\eta)^2}{(\tilde \lambda_1-1 )}\Bigg)
		\end{align*}
		Then
		\begin{align}
		\lambda_{1} \left(X_{T_0 \cup T_1}^t X_{T_0 \cup T_1}  \right) & \le \tilde \lambda_{s_0}-s_1 \varepsilon_{min} \label{lemme4:1}
		\end{align}
		and 
		\begin{align}
		\lambda_{s_0+s_1} \left(X_{T_0 \cup T_1}^t X_{T_0 \cup T_1}  \right) & \geq \tilde \lambda_1 + s_1\varepsilon_{max} \label{lemme4:2}
		\end{align}
	\end{lemm}
	\begin{proof}
		The proof relies on induction.
		First of all, note that from assumption \eqref{lemme4:h3}
		\begin{enumerate}
			\item[(i)] $\tilde \lambda_{s_0}-s_1 \varepsilon_{min}>\eta$;
			\item[(ii)] $\tilde \lambda_1+s_1\varepsilon_{max}<2-\eta$.
		\end{enumerate}
		We apply lemma \ref{lemme:2} to $X_{T_0}X^t_{T_0}+X_{j_1}X^t_{j_1}$ with $j_1 \in T_1$. We have
		\begin{align*}
		\lambda_{s_0+1}(X_{T_0}X^t_{T_0}+X_{j_1}X^t_{j_1})\geq \tilde\lambda_{s_0}-\varepsilon_{s_0,min}
		\end{align*}
		with $\varepsilon$ defined in \ref{lemme:2}. Since $\varepsilon_{s_0,min} \leq \varepsilon_{min}$, we get 
		\begin{align*}
		\lambda_{s_0}(X_{T_0}X^t_{T_0})\geq\lambda_{s_0+1}(X_{T_0}X^t_{T_0}+X_{j_1}X^t_{j_1})\geq \tilde\lambda_{s_0}-\varepsilon_{min}.
		\end{align*}
		It implies by (i) that
		\begin{align*}
		1-(s_0+s_1)\mu > \lambda_{s_0+1}(X_{T_0}X^t_{T_0}+X_{j_1}X^t_{j_1})> \eta.
		\end{align*}
        Thus, the induction hypothesis is verified and we can apply Lemma \ref{lemme:2} for the next step of the induction. This leads to
		\eqref{lemme4:1}.
		
		For the lower bound \eqref{lemme4:2}, we have from lemma \ref{lemme:3}
		\begin{align*}
		\lambda_1(X_{T_0}X^t_{T_0}+X_{j_1}X^t_{j_1}) \le \tilde \lambda_1 + \varepsilon_{s_0,max}
		\end{align*}
		Since $\varepsilon_{s_0,max} \leq \varepsilon_{max}$, we have by (ii) that
		\begin{align*}
		1<\lambda_1(X_{T_0}X^t_{T_0}+X_{j_1}X^t_{j_1}) <2-\eta.
		\end{align*}
		We can then apply lemma \ref{lemme:3} to the next step. The result follows by induction.
	\end{proof}
	\begin{coro}
		\label{cor}
		Let $T_0\subset \{1,\ldots,p\}$ with $\vert T_0 \vert=s_0$ and $X_{T_0}$ a submatrix of $X$. Let  $\lambda_1\geq ... \geq \lambda_{s_0}$ be the eigenvalues of $X_{T_0}X^t_{T_0}$. 
		Let 
		$\tilde \lambda_1 \ge \lambda_1$ and $\tilde \lambda_{s_0} \le \lambda_{s_0}$. Let $T_1\subset \{1,\ldots,p\}$ with $\vert T_1 \vert=s_1$ and $T_0 \cap T_1=\emptyset$. Set $\eta=\frac12$ and $s_1=3s_0$. Assume 
		\begin{enumerate}
			\item $1-(s_0+s_1)\mu>\tilde \lambda_{s_0}> \eta$;
			\item $1<\tilde \lambda_1<2-\eta$;
			\item $s_1< \min\left(\frac{\tilde \lambda_{s_0}-\eta}{\varepsilon_{min}},\frac{2-\eta-\tilde \lambda_1}{\varepsilon_{max}}\right)$; \label{lemme4:h3}
		\end{enumerate}
	    with 
		\begin{align*}
		\varepsilon_{min}=\frac14 \frac{s^3_0 \ \mu^2 \ /4+s_0^{3/2} \ \mu }{ \ \left(1-s_0 \mu^2- \frac12\right)}
		\end{align*}
		\begin{align*}
		\varepsilon_{max}=\frac14 \Bigg(\frac{144 \ s_0^4\mu^2+32 \ s_0^{3/2}\mu (2-\eta)^2}{(\tilde \lambda_1-1 )}\Bigg)
		\end{align*}
		Assume also
		\begin{align*}
		\mu \le \min \left\{\frac{1}{\sqrt{288 s_0^{5/2}\left(2 s_0^{3/2}+1\right)}},\frac{1}{\sqrt{\frac32 s_0^4 +6 s_0^{5/2}+2 s_0}}\right\}.
		\end{align*}
		Then,
		\begin{align}
		\lambda_{1} \left(X_{T_0 \cup T_1}^t X_{T_0 \cup T_1}  \right) & \le \tilde \lambda_1 + 3 s_0 \ \varepsilon_{max}
		\label{bornsup}
		\end{align}
		and 
		\begin{align}
		\lambda_{s_0+s_1} \left(X_{T_0 \cup T_1}^t X_{T_0 \cup T_1}  \right) & \geq \tilde \lambda_{s_0}-3s_0 \ \varepsilon_{min}.
		\label{borninf}
		\end{align}
		\end{coro}
			
		\begin{proof}
			Set $\eta=\frac12 $, assumption \eqref{lemme4:h3} writes 
			\begin{align*}
			s_1 < \frac{4(\tilde \lambda_{s_0}-\frac12)(\frac12-s_0 \mu^2)}{s_0^3 \mu^2 \frac14 +s_0^{3/2}\mu}
			\end{align*}
			and 
			\begin{align*}
			s_1 <  \frac{4(\frac32-\tilde  \lambda_1)(\tilde  \lambda_1-1)}{(s_0+s_1)^3 \mu^2 \frac94 + 9(s_0+s_1)^{3/2}\mu}
			\end{align*}
			which leads to 
			\begin{align*}
			s_1\left(s_0^3 \mu^2 \frac14 +s_0^{3/2}\mu\right)+4\left(\tilde \lambda_{s_0}-\frac12\right)s_0\mu^2<2\left(\tilde \lambda_{s_0}-\frac12\right)
			\end{align*}
			and
			\begin{align*}
			s_1\left((s_0+s_1)^3 \mu^2 \frac94 + 9(s_0+s_1)^{3/2}\mu\right)<4\left(\frac32-\tilde  \lambda_1\right)(\tilde \lambda_1-1).
			\end{align*}
			Since $\mu<1$
			\begin{align*}
			s_1\left(s_0^3 \mu^2 \frac14 +s_0^{3/2}\mu^2\right)+4\left(\tilde  \lambda_{s_0}-\frac12\right)s_0\mu^2<2\left(\tilde  \lambda_{s_0}-\frac12\right)
			\end{align*}
			and
			\begin{align*}
			s_1\left((s_0+s_1)^3 \mu^2 \frac94 + 9(s_0+s_1)^{3/2}\mu^2\right)<4\left(\frac32-\tilde \lambda_1\right)(\tilde \lambda_1-1)
			\end{align*}
			The result follows by factoring out $\mu^2$ and setting $s_1=3s_0$.
		\end{proof}

	\subsection{Bounding scalar products}
	\begin{lemm} 
		Let $|T_0|=s_0$ and $|T_1|=s_0$, $T_0$, $T_1$ disjoint. Let $T=T_0 \cup T_1$ and $T'$ be two disjoint subsets of $\{1,\ldots,p\}$ with  $|T'|=2s_0$. Let $g$ and $h$ be vectors in $\mathbb R^p$. Assume that 
	    \begin{align*}
    		\mu \le \min \left\{\frac{1}{\sqrt{288 s_0^{5/2}\left(2 s_0^{3/2}+1\right)}},\frac{1}{\sqrt{\frac32 s_0^4 +6 s_0^{5/2}+2 s_0}}\right\}.
		\end{align*}
		Then, 
		\begin{align}
		\left|\langle X_T g_T,X_{T'}h_{T'} \rangle \right| & \le \left( \lambda_1 + 3\ s_0 \ \varepsilon_{max} \right) \ \Vert g_T\Vert_2 \ \Vert h_{T'} \Vert_2.
		\end{align} 
		\label{scal}
	\end{lemm} 
	\begin{proof}
		Assume first that $\Vert g_{T} \Vert_2=\Vert h_{T'} \Vert_2=1$. 
		The parallelogram law now gives 
		\begin{align*}
		\left\vert \langle X_T g_T, X_{T'} h_{T'} \rangle \right\vert & \le
		\frac14 \left\vert \Vert X_T g_T + X_{T'} h_{T'} \Vert_2^2
		-\Vert X_T g_T - X_{T'} h_{T'} \Vert_2^2 \right\vert \\
		& \le  \frac14 \left\vert \Vert X_T g_T + X_{T'} h_{T'} \Vert_2^2
		-\Vert X_T g_T - X_{T'} h_{T'} \Vert_2^2 \right\vert \\
		\end{align*}
		Notice that 
		\begin{align*} 
		\Vert X_T g_T\pm X_{T'} h_{T'} \Vert_2^2 & = \Vert X_{T \cup T'} (g_{T} \pm h_{T'}) \Vert_2^2.
		\end{align*} 
		By Corollary \ref{cor}, we have 
		\begin{align*} 
		\left( \lambda_{s_0} - 3\ s_0 \ \varepsilon_{min} \right) \Vert g_T+h_{T'} \Vert_2^2 \ \le \Vert g_T+h_{T'} \Vert_2^2 
	    & \Vert X_T g_T\pm X_{T'} h_{T'} \Vert_2^2  \le \left( \lambda_1 + 3\ s_0 \ \varepsilon_{max} \right) \ \Vert g_T+h_{T'} \Vert_2^2. 
		\end{align*} 
        From this, and the fact that $g_T$ and $h_T$ are unit norm, we deduce that 
		\begin{align*}
		\left\vert \langle X_T g_T, X_{T'} h_{T'} \rangle \right\vert & \le
		\lambda_1-\lambda_{s_0}  + 3\ s_0 \ \left(\varepsilon_{max} +\varepsilon_{min}\right). 
		\end{align*}
		The proof is completed using homogeneity. 
	\end{proof}
	
	%
	%


	
	
	
	\bibliographystyle{amsplain}
	\bibliography{CohWeakNSP}
	
\end{document}